\documentclass[11pt,letterpaper]{amsart}

\usepackage{mathtools, thm-restate}
\usepackage{amsmath}
\usepackage{amsthm}
\usepackage{subcaption}
\usepackage[inline]{enumitem}
\usepackage{amssymb}
\usepackage{mathrsfs}
\usepackage{tikz-cd}
\usepackage{import}
\usepackage{pdfpages}
\usepackage{transparent}
\usepackage{xcolor}
\usepackage[colorlinks=true]{hyperref}
\usepackage{cleveref}
\usepackage{comment}
\usepackage{float}
\usepackage{pinlabel}

\newtheorem{thm}{Theorem}
\numberwithin{thm}{section}
\newtheorem{lem}[thm]{Lemma}
\newtheorem{prop}[thm]{Proposition}
\newtheorem{cor}[thm]{Corollary}
\newtheorem{conj}[thm]{Conjecture}

\theoremstyle{definition}
\newtheorem{dfn}[thm]{Definition}
\newtheorem{rmk}[thm]{Remark}
\newtheorem{ex}[thm]{Example}
\crefname{ex}{example}{examples}

\newcommand{\Z}{\mathbb{Z}}

\newcommand{\modg}{\mathrm{Mod}_g}

\newcommand{\mir}[1]{{\color{purple}MK: #1}}

\begin{document}
\title{Heegaard Floer homology and the word metric on the Torelli group}

\author{Santana Afton}
\email{santana.afton@gmail.com}

\author{Miriam Kuzbary}
\address{Department of Mathematics and Statistics, Amherst College, Amherst, MA 01002}
\email{mkuzbary@amherst.edu}

\author{Tye Lidman}
\address{Department of Mathematics, North Carolina State University, Raleigh, NC 27607}
\email{tlid@math.ncsu.edu}

\begin{abstract}
We study a relationship between the Heegaard Floer homology correction terms of integral homology spheres and the word metric on the Torelli group.  For example, we use it to give a quick proof that the Cayley graph of the Torelli group has infinite diameter in the word metric induced by the generating set of all separating twists and bounding pair maps.  On the other hand, we show that many subsets of the Torelli group are bounded with respect to this metric. Finally, we address the case of rational homology spheres by ruling out a certain Morita-type formula for congruence subgroups of mapping class groups. 
\end{abstract}
\maketitle

\section{Introduction} \label{sec:intro} Recall that for every closed, connected, oriented $3$-manifold $Y$ there is a positive integer $g$ such that $Y$ can be decomposed into two genus $g$ handlebodies glued together along an orientation reversing diffeomorphism of their boundary. This is called a \textit{Heegaard splitting} of $Y$, and is unique up to isotopy and stabilizations by a classical theorem of Reidemeister and Singer.  
As used to great effect in \cite{birmancraggs}, any Heegaard splitting of $Y$ can be constructed from a Heegaard splitting of the $3$-sphere in the following way. Given a Heegaard splitting of $S^3$ with gluing map $f$, there is an element $\phi$ in the mapping class group of the surface $\Sigma_g$ (denoted $\mathrm{Mod}_g$ for a genus $g$ surface) such that $Y$ is the result of gluing handlebodies together by the composition $f \circ \phi$. The equivalence of Heegaard splittings up to isotopies and stabilizations can be translated into an algebraic statement about mapping class groups \cite{birmancraggs, pitsch}.  More generally, given any oriented $3$-manifold $Y$ and a smoothly embedded surface $f: \Sigma_g \hookrightarrow Y$ in it, we can cut $Y$ open along $S=f(\Sigma_g)$ and reglue by some diffeomorphism $f \circ \phi$ of $S$. We will refer to this $3$-manifold as $Y_{S,\phi}$. If the surface $S$ is clear from context, we will omit $S$ from the notation.

This observation relating Heegaard splittings and the mapping class group leads to many results relating the algebraic structure of $\mathrm{Mod}_g$ and the topology of $3$-manifolds. For example, in \cite{lickorish}, Lickorish proves that every $3$-manifold can be obtained by surgery on a link; this is a direct application of Dehn's result that the mapping class group is generated by Dehn twists. Moreover, there is an entire program of using invariants of $3$-manifolds, such as the Rokhlin invariant $\mu$, to explore the structure of the {\em Torelli group}, which is the subgroup of $\mathrm{Mod}_g$ acting trivially on the homology of the surface $\Sigma_{g}$. This subgroup, denoted $\mathcal{I}_g$, is the main object of study in this paper. 

We will briefly survey what is known about the connections between $3$-manifold invariants and the Torelli group. The aforementioned work by Birman-Craggs \cite{birmancraggs} shows that $\phi \mapsto \mu(S^3_{S,\phi})$ provides homomorphisms from $\mathcal{I}_g$ to $\mathbb{Z}/2$; see also the work of Johnson \cite{johnson}. Furthermore, by lifting the $\Z/2\Z$-valued Rokhlin invariant to the $\Z$-valued Casson invariant $\lambda$, the Birman-Craggs homomorphism lifts to a $\mathbb{Z}$-valued homomorphism on a certain subgroup of the Torelli group by \cite{moritacharclass}.  More recently, in \cite{MSS} the Casson invariants of these glued up $3$-manifolds were instrumental in detecting the differences between two important algebraic filtrations on the Torelli group.  


In this work, we focus primarily on integral homology spheres as these are exactly the $3$-manifolds corresponding to changing the monodromy of a Heegaard splitting of $S^3$ by an element of the Torelli group \cite{moritacasson}. More generally, we will refer to $Y_{S, \phi}$ when $\phi \in \mathcal{I}_g$ as \textit{Torelli surgery}. A natural question is what Torelli surgery can see about the geometry of the Torelli group.  For example, Broaddus-Farb-Putman showed in \cite{broaddusfarbputman}  that for a Heegaard surface $S$ in $Y$ there is a quadratic bound
\begin{equation}\label{eq:bfp}
|\lambda(Y_{S,\phi}) - \lambda(Y)| \leq C \| \phi\|^2,
\end{equation}
where $\| \phi \|$ denotes the word metric with respect to a finite generating set on the Torelli group and $C$ is a fixed constant that depends on this choice.  Since Heegaard Floer homology (and other Floer theories) can recover the Casson invariant, we were curious what could be learned from applying these invariants to Torelli surgery.  We focus here on the Heegaard Floer $d$-invariants from \cite{ozszdinvt} as they are both important numerical 3-manifold invariants in their own right and because they provide a direct connection between Floer homology and the Casson invariant, which we recall in Equation~\eqref{eq:HFandcasson} below.

Our first result is the following.  Consider $\phi$ a separating Dehn twist or bounding pair map on a surface of genus $g$  (for more details on these mapping classes, see the next section).  By \cite{birman}, these maps generate $\mathcal{I}_g$. In either case, the curves which are being twisted separate the surface into two components and for each $\phi$ we will denote the minimum genus among these two surfaces as $k_\phi$.  Note that $k_\phi \leq \lfloor \frac{g}{2} \rfloor$.  

\begin{thm}\label{thm:lineargrowth} Let $S$ be an embedded surface of genus $g$ in a homology sphere $Y$.  Let $A$ be a generating set for $\mathcal{I}_g$ consisting of bounding pair maps and separating twists.  Let $k_A = \max_{\phi \in A} k_\phi$. Note that a generating set for $\mathcal{I}_g$ can always be chosen such that $k_A = 1$.  For any $\phi \in \mathcal{I}_g$, 
\[ |d(Y) - d(Y_{S,\phi})| \leq 2 \left \lceil \frac{k_A}{2} \right \rceil ||\phi||_A\]
where $||\phi||_A$ is the word length of $\phi$ in $\mathcal{I}_g$ and $d$ is the Heegaard Floer d-invariant. 
\end{thm}

A more general version of the above theorem is given in Section \ref{sec:proofs}.   We will use the notation $\| \cdot \|_\infty$ to denote the word length on the Torelli group with respect to the infinite generating set of all separating twists and bounding pair maps. 

An immediate consequence of Theorem~\ref{thm:lineargrowth} is a quick proof of the following result, which is likely known to experts:
\begin{cor}\label{cor:surgery-bound}
Suppose that $K$ is a knot in a homology sphere $Y$ and $n \in \mathbb{Z}$.  Then, 
\[
|d(Y_{1/n}(K)) - d(Y)| \leq 2|n| \left \lceil \frac{g(K)}{2} \right \rceil. 
\]
\end{cor}
\begin{proof}
Given a Seifert surface $F$ for $K$, $F \times I$ is a genus $2g$ handlebody and $K$ sits on the boundary as a separating curve, splitting the boundary into two genus $g$ surfaces.  Now, by the Lickorish trick (see Lemma~\ref{lem:twist-surgery}  or Example~\ref{ex:knottotwist}) below, $1/n$-surgery on $K$ can be recast as Torelli surgery using powers of a Dehn twist along this separating curve.  The result then follows from Theorem~\ref{thm:lineargrowth}.     \end{proof}

\begin{rmk}
While this corollary suggests that the $d$-invariants of $1/n$-surgery can be unbounded for a fixed knot, this is impossible.  This follows, for example, from Theorem~\ref{thm:family} below (see also Proposition~\ref{prop:separating} for this specific case).    
\end{rmk}  

We can use the well-known result that the $d$-invariants are unbounded on manifolds with fixed Heegaard genus to prove the following corollary about the Torelli group. We thank Andy Putman for explaining that this would be worthy of proof. 

\begin{restatable}{corol}{cayleygraph}\label{cor:cayleygraph}
The Cayley graph of the Torelli group has infinite diameter with respect to the infinite generating set consisting of all separating twists and bounding pair maps. 
\end{restatable}

\begin{proof}
Fix $g \geq 2$.  Consider the Brieskorn spheres $Y_n = \Sigma(2,4n+1,8n+1)$, oriented as the boundary of a positive-definite plumbing; these are alternatively described as $+1$-surgery on the $(2,4n+1)$-torus knot.  The manifold $Y_n$ has $d$-invariant $-2n$ (see for example \cite[Corollary 1.5]{OzsvathSzabo-alternating}).  Since the Heegaard genus of $Y_n$ is 2, after stablization these manifolds have Heegaard splittings of genus $g$.  Now, fix a Heegaard surface $S$ in $S^3$, and choose $\phi_n \in \mathcal{I}_g$ such that $S^3_{\phi_n} = Y_n$.  Theorem~\ref{thm:preciselineargrowth} shows that $\| \phi_n \|_\infty$ is unbounded in $n$.  This proves the desired result. 
\end{proof}


\begin{rmk}
Ian Agol noted to us that this result can alternatively be proved from the fact the stable commutator length of a Dehn twist is positive (see \cite{BestvinaBrombergFujiwara, EndoKotschick, Korkmaz}).  It is interesting to note that the methods of \cite{EndoKotschick, Korkmaz} rely on the topology of symplectic four-manifolds, but the four-dimensional information captured using these techniques is different than what the $d$-invariants seem able to capture. 
\end{rmk}

\begin{rmk}
It seems likely that one could also study various subgroups of the braid group as follows.  Given a braid which lifts to a Torelli element on a branched cover, we can consider the $d$-invariant of Torelli surgery on say $S^3$.  We do not pursue this here.
\end{rmk}


 %

For the sake of comparison, we show that there is no analogue of \eqref{eq:bfp} when working with $\| \cdot \|_\infty$.  

\begin{restatable}{thm}{cassonfail}\label{thm:cassonfail}
Fix an integral homology sphere $Y$ and an embedded surface $S$ of genus $g$.  There is no function $f: \mathbb{R}^+ \to \mathbb{R}^+$ such that for all $\phi \in \mathcal{I}_g$,
\[
|\lambda(Y) - \lambda(Y_{S,\phi})| \leq f(\|\phi\|_\infty).
\]
\end{restatable}

Another consequence of Theorem \ref{thm:lineargrowth} is in the setting of the homology cobordism group; this is the set of homology cobordism classes of integral homology spheres with connected sum as the group operation \cite{acuna}. Recall that two homology three-spheres are homology cobordant if they cobound a smooth four-manifold with the homology of $S^3 \times I$. By \cite{furuta, fintushelstern}, this group is infinitely generated; however, Theorem \ref{thm:lineargrowth} illustrates that we can use the fact that the Torelli group is finitely generated for $g>2$ as proven in \cite{johnson1985structure} to bound the $d$-invariant without a finite generating set for the homology cobordism group itself.  In the case of genus 2, Johnson proved the Torelli group is infinitely generated and Mess further showed that it is a free group on infinitely many separating twists \cite{mess}.

The explicit relationship between the $d$-invariant and the Casson invariant $\lambda$ is 
\begin{equation} \label{eq:HFandcasson}
    \lambda(Y) = \chi(HF_{red}(Y)) - \frac{1}{2}d(Y)
\end{equation}
by Theorem 1.3 in \cite{ozszdinvt} where $HF_{red}(Y)$ is the $U$-torsion summand of $HF^{-}(Y)$ and the Casson invariant is normalized so that $\lambda(\Sigma(2,3,5))=1$ where $\Sigma(2,3,5)$ is oriented as $+1$-surgery on the right-handed trefoil.   In fact, the $d$-invariant was originally called the ``correction term" precisely because it captures the difference between the Casson invariant and $\chi(HF_{red}(Y))$. Since Broaddus-Farb-Putman show that \eqref{eq:HFandcasson} is sharp in a suitable sense, combining \eqref{eq:HFandcasson} with Theorem \ref{thm:lineargrowth}, we see this means the Euler characteristic of $HF_{red}(Y)$ can grow much faster than $d(Y)$ as a function of the word length in the Torelli group.

To further account for the difference between these Heegaard Floer invariants and previously studied $3$-manifold invariants in this setting, we prove the following theorem. Recall that both the Rokhlin invariant and the Casson invariant are finite-type invariants \cite{ohtsuki}.

\begin{restatable}{thm}{finitetype}\label{thm:finitetype}
Neither the $d$-invariant nor the Euler characteristic of $HF_{red}$ is a finite-type invariant of homology three-spheres.
\end{restatable}

We can also sometimes establish much stronger bounds on the $d$-invariants than Theorem~\ref{thm:lineargrowth}.  One example is the following theorem.
\begin{restatable}{thm}{family}\label{thm:family}
Let $A=\{\phi_1,\ldots, \phi_k\}$ be a collection of separating Dehn twists and bounding pair maps on a parameterized embedded surface $S$ of genus $g$ in a homology sphere $Y$. Consider the subset $A^{\prime}$ of $\mathrm{Mod}_g$  consisting of length at most $q$ products of powers of elements of $A$, i.e. elements of the form $\phi_{i_1}^{n_1} \ldots \phi_{i_q}^{n_q}$. Then, the set $\{d(Y_{a}) \mid a \in A^{\prime} \}$ is bounded.  
\end{restatable}

Theorem~\ref{thm:family} further implies that if $H$ is an abelian group generated by separating twists and bounding pair maps, then $\{d(Y_\phi) \mid \phi \in H\}$ is bounded.  To incentivize more interest in these questions, we pose the following optimistic conjecture:

\begin{conj}
Let $S$ be a genus $g$ surface in a homology sphere $Y$ and let $H$ be an abelian subgroup of Torelli.  Then $\{d(Y_\phi) \mid \phi \in H\}$ is bounded.      
\end{conj}

Underneath these results is the fact that, in general, the result of taking the connected sum of $Y_{\phi}$ and $Y_{\psi}$ is not homeomorphic to $Y_{\phi \circ \psi}$. 
However, the Casson invariant is ``almost'' additive when $S$ is a Heegaard surface, as the Morita formula (Theorem 4.3 in \cite{moritacasson}) gives: 
\begin{equation}
    \label{eq:moritaformula}
\lambda(Y_{\phi \circ \psi}) = \lambda(Y_\phi) + \lambda(Y_\psi) + 2\delta(\phi, \psi), 
\end{equation}

\noindent where $\phi, \psi \in \mathcal{I}_g$ and $\delta(\phi, \psi)$ is a term depending on the image under the Johnson homomorphism of $\phi, \psi$ considered as elements of $\mathcal{I}_{g,1}$ (and hence independent of the choice of embedding of the Heegaard surface). This formula was used to prove the aforementioned theorem of Broaddus, Farb, and Putman.  

While the Casson invariant can be extended to all closed, oriented three-manifolds \cite{walker, lescop}, we disprove the existence of a certain extension of the Morita formula to rational homology spheres.  Just as integral homology spheres can be constructed as $Y_{S, \phi}$ where $\phi \in \mathcal{I}_g$, for $S$ a Heegaard surface, we can characterize rational homology spheres using specific subgroups of $\modg$ in the following way.  Recall that for an integer $k \geq 2$, the level $k$ congruence subgroup $\modg(k)$ consists of the mapping classes that act trivially on the homology of $\Sigma_g$ with $\mathbb{Z}/k\mathbb{Z}$ coefficients, and one important property the congruence subgroups have distinct from those of $\mathcal{I}_g$ is that each level k congruence subgroup is finite index in $\modg$.  By Theorem 1.1 in \cite{pitschriba} (and related work in \cite{cooper}) every rational homology sphere has gluing map in the level $p$ congruence subgroup $\modg(p)$ for some prime number $p$. 

Therefore, it is a natural question to ask whether the Morita formula for the Casson invariant of integral homology spheres corresponding to products of elements in $\mathcal{I}_g$ generalizes to a formula for the Casson-Walker invariant of rational homology spheres corresponding to products in $\modg(k)$ for each fixed $k$ (note that again by Theorem 1.1 in \cite{pitschriba} this is exactly rational homology spheres $Y$ with $|H_1(Y)|=n$ where $k$ divides $n-1$ or $n+1$). By Lemma 5.6 in \cite{cooper}, we know every element of $\modg(k)$ factors as a product of a Torelli element and a product of $k^{th}$ powers of Dehn twists. In the following theorem, we use elements written in this form to show there cannot be a Morita formula for homology spheres with gluing map in $\modg(k)$ for any $k>1$.

\begin{restatable}{thm}{brokenmorita}\label{thm:brokenmorita}
Let $k$ be an arbitrary non-zero integer.  There exists a Heegaard surface $S$ in $S^3$, Torelli elements $\psi, \eta$, and a non-separating Dehn twist $\phi$ so that: 
\begin{enumerate}
\item $S^3_{\phi^k}$ is an integer homology sphere; 
\item There is an automorphism of $S$ taking $\psi$ to $\eta$ and $\phi$ to $\phi$; 
\item $\lambda(S^3_{\psi \phi^k}) - \lambda(S^3_{\psi}) - \lambda(S^3_{\phi^k})$ is not equal to $\lambda(S^3_{\eta \phi^k}) - \lambda(S^3_{\eta}) - \lambda(S^3_{\phi^k})$.  
\end{enumerate}
Consequently, there is no Morita formula for $\modg(k)$, that is, a formula for the Casson (or Casson-Walker) invariant of $Y_{S, \phi \circ \psi}$ where $ \phi, \psi \in \modg(k)$ that is independent of the embedding of a Heegaard surface S.  
\end{restatable}

\begin{rmk} This is quite different from the Perron conjecture \cite{perron}, which tries to define a Casson invariant from an element $\psi \cdot \prod_i \phi_i^k \in Mod_g(k)$ by reducing $\lambda(S^3_\psi)$ mod $k$.
\end{rmk}

We end this introduction with some discussion about the Johnson kernel.  The kernel of the Johnson homomorphism, denoted $\mathcal{K}_g$, is exactly the subgroup of $\mathcal{I}_g$ generated by separating twists, and Morita showed every integral homology sphere is actually $Y_{\phi}$ for some $\phi \in \mathcal{K}_g$. Thus, the Casson invariant is additive on $\mathcal{K}_g$. 

\begin{rmk}
The same argument as in Corollary \ref{cor:cayleygraph} shows that the Johnson kernel has infinite diameter with respect to the generating set consisting of all separating twists.
\end{rmk}


Though we do not have a Morita-type formula for the $d$-invariant of the result of gluing by a product of mapping classes, we observe the following consequence of combining \eqref{eq:HFandcasson} with the Morita formula \eqref{eq:moritaformula}.

\begin{rmk}Let $\phi, \psi \in \mathcal{K}_g$. Then 
\[d(Y_{\phi \circ \psi}) = d(Y_{\phi})+d(Y_{\psi})+2(\chi(HF_{red}(Y_{\phi \circ \psi}))-\chi(HF_{red}(Y_{\phi}))-\chi(HF_{red}(Y_{\psi}))).\]
\end{rmk}

When any of the constituent $3$-manifolds $Y_{\phi \circ \psi}$, $Y_{\phi}$, and $Y_{\psi}$ have torsion-free $HF^{-}$, the formula becomes much simpler. However, it is an open question whether there are any such integral homology spheres other than $S^3$ and connected sums of the Poincar\'{e} homology sphere.  Similarly, it is natural to ask whether there is any subgroup of $\mathcal{K}_g$ on which the $d$-invariant is a group homomorphism, but the question is beyond the scope of this paper. Note that it is not even a quasimorphism on the Johnson kernel by applying the proof of Corollary \ref{cor:cayleygraph} along with the fact that every homology sphere has a gluing map in the Johnson kernel \cite{moritacasson}.

\begin{rmk}
It seems likely that one could do analogous work with the Torelli group (or at least Johnson kernel) for surfaces in other three-manifolds.  For example, one could consider the collection of $d$-invariants of Torelli surgery on a Heegaard splitting of a rational homology sphere or the totally twisted $d$-invariant \cite{BehrensGolla} applied to the mapping torus of a Torelli element.   
\end{rmk}




\subsection{Outline} Section \ref{sec:background} lays the technical foundations for this work, followed by proofs of Theorem~\ref{thm:lineargrowth} and Theorem~\ref{thm:cassonfail}  in Section \ref{sec:proofs}. Section \ref{sec:bounded} details the proof of the boundedness result Theorem~\ref{thm:family}. Finally, in Section~\ref{sec:rational}, we consider the extension to rational homology spheres and prove Theorem~\ref{thm:brokenmorita}.  

\subsection{Acknowledgments} MK was partially supported by NSF Award No. DMS-2103142. TL was partially supported by NSF grant DMS-1745583. TL benefited greatly from being in residence at SL Math during Fall 2022 (NSF grant DMS-1440140) where some of this work took place.  This material is based upon work supported by the National Science Foundation under Grant No. DMS-1928930, while MK and TL were in residence at SL Math in Berkeley, California, during the Summer Research in Mathematics program in 2023.  The authors would like to thank Ian Agol, Tara Brendle, Benson Farb, Jennifer Hom, Brendan Owens, and Sam Taylor for many fruitful discussions and \.{I}nan\c{c} Baykur for comments on an earlier draft. Special thanks to Andrew Putman for many conversations about this work and for introducing the second author to the Torelli group in a wonderful special topics class some years ago. Finally, thank you to the reviewer for their thoughtful comments.

\section{Background}\label{sec:background}

As the intended audience for this work includes mathematicians interested in mapping class groups, 3-manifold topology, and Heegaard Floer homology, we have included a brief survey of some  relevant background. Throughout this article, we work with a fixed orientable genus $g$ surface $\Sigma_g$.  By abuse of notation, we will say that $S$ is an embedded genus $g$ surface in a homology sphere $Y$ to mean we fix an embedding $f: \Sigma_g \to Y$ and $S = f(\Sigma_g)$ as in the introduction; consequently, an element of $\modg$, the mapping class group of $\Sigma_g$, induces a well-defined mapping class of $S$.  

\subsection{The Torelli group}
Recall that $\mathrm{Mod}_g$ is generated by Dehn twists along simple closed curves in $\Sigma_g$.  In the interest of matching the sign conventions in \cite{farbmargalit}, we will consider a left-handed twist along a simple closed curve to be a positive twist. This corresponds to the standard right handed orientation of $S^1$ and the parametrization of a Dehn twist inside the annulus as $T_c: S^1 \times [0,1] \rightarrow S^1 \times [0,1]$ with $T_c(\theta, t)=(\theta+2\pi t, t)$ .

For the Torelli group, $\mathcal{I}_g$ we can also write down a generating set, but we need more terminology first.  First, a {\em separating twist} is a Dehn twist along a separating curve, which lives in the Torelli group.  Next, we say $\phi \in \mathrm{Mod}_g$ is a {\em bounding pair map} if it can be written as the product of a positive twist $T_a$ on a curve $a \subset \Sigma_g$ and a negative twist $T_b^{-1}$ on a curve $b \subset \Sigma_g$ where $a$ and $b$ cobound a subsurface of $\Sigma_g$.  As discussed above, for $g = 2$, Mess proved that $\mathcal{I}_g$ is an infinitely generated free group which has a collection of separating twists as generators; for $g > 2$, bounding pair maps and separating twists generate the Torelli group $\mathcal{I}_g$ \cite{birman} \cite{powell}.  Note that if we allow one curve in a bounding pair map to be trivial, then every separating twist can technically be viewed as a bounding pair map.\footnote{For the purposes of exposition and to align with perspectives in geometric group theory, we will prove results for separating twists and bounding pair maps separately, when often it might seem more succinct to give a single proof for this more general notion of bounding pair maps.  We feel it is beneficial for the reader to study separating twists separately as a warm-up to build intuition.} By work in \cite{johnson1985structure}, $\mathcal{I}_g$ is finitely generated when $g>2$.  The subgroup generated by separating twists is the {\em Johnson kernel} $\mathcal{K}_g$.  

From now on, we will write $\gamma$ to denote either a simple closed curve when we are performing a Dehn twist along a single curve or a curve with two simple closed components if we are referring to a bounding pair map and we will write  $\gamma = \gamma_1 \cup \gamma_2$.  In the latter case, we will use the  convention that $\gamma$ is ordered so that a positive bounding pair map refers to a positive Dehn twist along $\gamma_1$ and a negative Dehn twist along $\gamma_2$. 

We quickly mention two standard results that allow us to easily perform computations and analyze Dehn twists along more complicated curves.

\begin{prop}[E.g. Fact 3.7 in \cite{farbmargalit}]\label{prop:twistatwist}
For any $f\in \mathrm{Mod}_g$ and any isotopy class $a$ of simple closed curves in $S$ we have
\[ T_{f(a)}=fT_af^{-1}\]
\end{prop}

\begin{prop}\label{prop:conjugate}
Dehn twists along any two non-separating simple closed curves on $\Sigma_g$ are conjugate.  Dehn twists along separating simple closed curves $\gamma_1, \gamma_2$ on $\Sigma_g$ are conjugate if and only if $\Sigma_g - \gamma_1$ and $\Sigma_g - \gamma_2$ are homeomorphic. 
\end{prop}

\subsection{Dehn twists and Dehn surgeries}
To set some more notation, as in the introduction, $Y_{S,\phi}$ denotes surgery on the embedded surface $S \subset Y$ with gluing map $\phi \in \mathrm{Mod}_g$. When $L$ is a nullhomologous link in a three-manifold $Y$, Dehn surgery on an $n$-component link $L\subset Y$ with surgery coefficients $\mathbf{r} \in \mathbb{Q}^n$ will be denoted $Y_{\mathbf{r}}(L)$. 

The key insight this work relies on is by Lickorish in \cite{lickorish}: the result of cutting open a $3$-manifold $Y$ along an embedded surface $\Sigma$ and regluing by a Dehn twist along a simple closed curve embedded in the surface is the same as performing Dehn surgery on that same curve instead thought of as a knot in $Y$. This is often called the Lickorish trick, and though it is well-known we will spell this out in Figure \ref{fig:lickorishtrick} and Lemma \ref{lem:twist-surgery} for the benefit of the reader.

\begin{figure}[h]
    \centering
     \begin{subfigure}{0.45\textwidth}
     \labellist
        \large\hair 2pt
        \pinlabel {{\color{teal}$\Sigma$}} [tr] at 0 140
        \pinlabel {{\color{orange}$\mu$}} [bl] at 80 -17
        \pinlabel {{\color{violet}$\gamma$}} [r] at 215 153
\endlabellist
        \centering
        \includegraphics[width=0.8\linewidth]{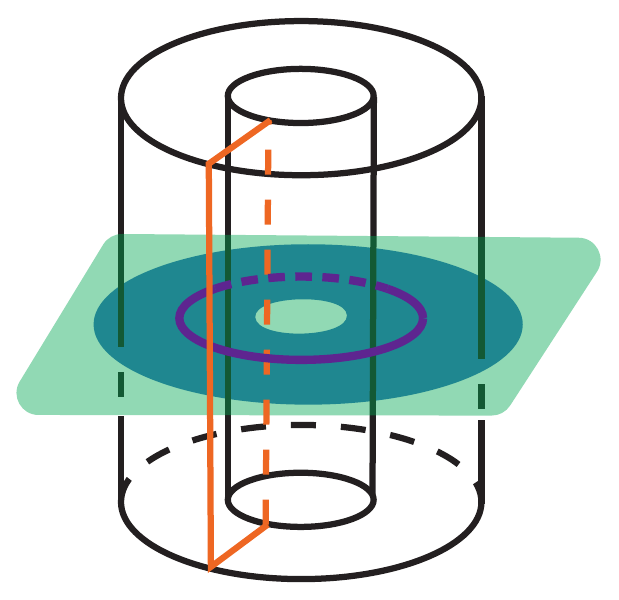}
    \end{subfigure}
    \begin{subfigure}{0.45\textwidth}
         \labellist
        \large\hair 2pt
        \pinlabel {{\color{orange}$\mu - \lambda$}} [bl] at 70 -17
        \endlabellist
        \centering
        \includegraphics[width=0.8\linewidth]{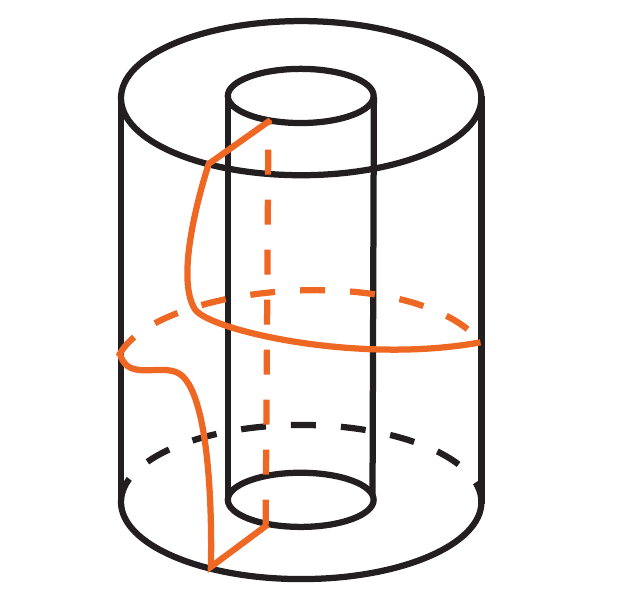}
    \end{subfigure}
    \caption{In this local picture of a curve $\gamma$ (in purple) lying on a surface $\Sigma$ (in green) embedded in a 3-manifold Y, we see the blue annulus is a neighborhood of the curve in the surface $\Sigma$, and the black solid torus is a neighborhood of $\gamma$ in $Y$. The meridian $\mu$ of $\gamma$ viewed as a knot in $Y$ is shown in orange. The result of cutting open $Y$ along $\Sigma$, performing a single positive Dehn twist on $\gamma \subseteq \Sigma$, and gluing back together along $\Sigma$ is exactly the same as removing the solid torus neighborhood of $\gamma$ and gluing it back with its meridian now going to the $\mu - \lambda$ curve shown on the right. This is precisely $-1$-surgery on $\gamma \subseteq Y$.}
    \label{fig:lickorishtrick}
\end{figure}

\begin{lem}\label{lem:twist-surgery}
Let $S$ be an embedded surface in a three-manifold $Y$ and let $\gamma$ be an oriented, simple closed curve on $S$ which is nullhomologous in $Y$.  Let $s$ denote the surface framing of $\gamma$ viewed as an integer.  Let $\phi$ denote the Dehn twist along $\gamma$.  Then, $Y_{S,\phi^n}$ can be described as $Y_{s - \frac{1}{n}}(\gamma)$.
\end{lem}
\begin{proof} Note that performing surgery on $S$ in this way fixes Y outside of a neighborhood of $\gamma \subset S$ considered as a subset of Y. Therefore, we only need to determine how $\nu(\gamma)$, a solid torus, changes after regluing $Y$ together by $\phi^n$. Modifying $Y$ inside a solid torus only is exactly the definition of Dehn surgery and can be completely described by the image of the meridian of $\nu(\gamma)$ after the diffeomorphism $\phi^n$ has been applied to $S$. Using coordinates on this torus corresponding to the meridian $\mu$ and the Seifert longitude $\lambda$, this means the meridian of $\nu(\gamma)$ will be sent to the curve $(1-ns)\mu-n\lambda$ where $s$ corresponds to the surface framing of $\gamma$ as in Figure \ref{fig:lickorishtrick}. 
\end{proof}



Since a smoothly embedded surface is bicollared, we can think about the result of cutting open $Y$ along $S$ and regluing by a composition $\phi \circ \psi$ of mapping classes in $\mathrm{Mod}_g$ as the same as cutting along $S \times \{0\}$ and regluing by $\psi$ and cutting along $S \times \{1\}$ and regluing by $\phi$.  
From this perspective, we see that there is a relationship between the Dehn twists in a factorization of $\phi$ and a surgery diagram for $Y_{S,\phi}$.  More precisely, if $\phi$ is given by $T_{\gamma_1}^{n_1} \ldots T_{\gamma_k}^{n_k}$, then $Y_{S,\phi}$ can be described by surgery on a $k$-component link in $Y$ where the individual components are isotopic to $\gamma_1,\ldots, \gamma_k$.  For example, in the case of a separating twist $\phi$, Lemma~\ref{lem:twist-surgery} shows that $Y_{\phi^n}$ can be described by surgery on $Y$ as in the left-hand side of Figure~\ref{fig:twist-surgery}.  Using handleslides and slam dunks, we obtain the alternative surgery descriptions in the right side of Figure~\ref{fig:twist-surgery}.

\begin{figure}
    \centering
      \begin{subfigure}{0.3\textwidth}
        \centering
        \includegraphics[width=0.8\linewidth]{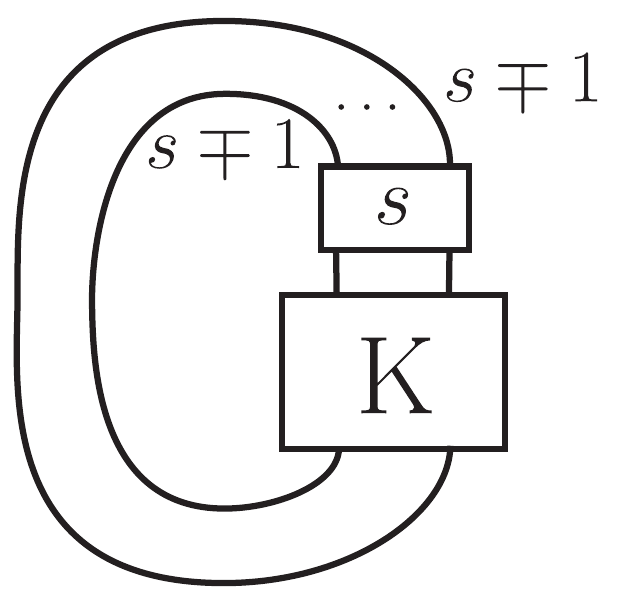}
    \end{subfigure}
     \begin{subfigure}{0.35\textwidth}
        \centering
        \includegraphics[width=0.8\linewidth]{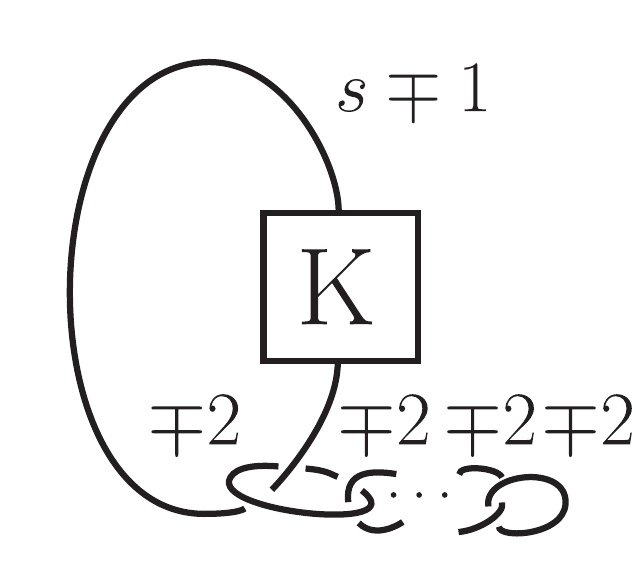}
    \end{subfigure}
    \begin{subfigure}{0.3\textwidth}
        \centering
        \includegraphics[width=0.8\linewidth]{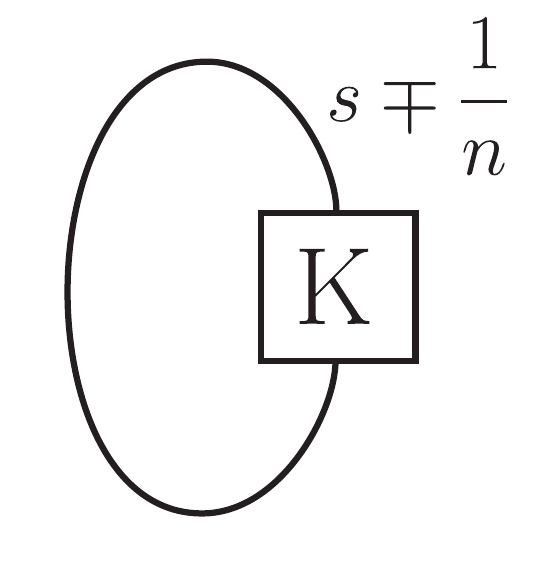}
    \end{subfigure}
    \caption{Above we see three equivalent surgery diagrams for the result of cutting open $S^3$ along a surface and twisting by $\phi^n$ where $\phi$ is a single Dehn twist along a simple closed curve in the surface isotopic to the knot $K$. In these diagrams, $s$ denotes the surface framing of the curve. In the left picture, n strands of the link run through the twist box and are tied into K.}
    \label{fig:twist-surgery}
\end{figure}

In general, if we instead apply Torelli surgery using a power of a product of various Dehn twists, the story becomes more complicated. By work in \cite{moritacasson}, every integral homology sphere can be obtained as $Y_{S,\phi}$ where $\phi \in \mathcal{I}_g$. In fact, every integral homology sphere can be obtained as $Y_{S,\phi}$ where $\phi \in \mathcal{K}_g$; however, it is  more difficult to find such an element in $\mathcal{K}_g$ compared to $\mathcal{I}_g$.  

\begin{lem}\label{lem:boundingpair-surgery}
Let $S$ be an embedded surface in a three-manifold $Y$ and let $\gamma = \gamma_1 \cup \gamma_2$ be a pair of oriented, simple closed curves on $S$ which together bound a subsurface of $S$.  Let $\phi$ denote the bounding pair map corresponding to $\gamma$.  
If $\gamma_1$, $\gamma_2$ are nullhomologous in Y, then $Y_{\phi}$ is $Y_{s-\frac{1}{n},s+\frac{1}{n}}(\gamma)$.  
\end{lem}
\begin{proof} Note that for a bounding pair map the surface framings on the two curves $\gamma_1$ and $\gamma_2$ are necessarily the same integers since the two curves are homologous.  Indeed, $lk(\gamma_1,\gamma_1^+)$ is the intersection of $\gamma_1$ with a Seifert surface $F$ for $\gamma_1$, which we can choose to intersect $S$ only at $\gamma_1$.  By gluing on the surface $S_0$ that $\gamma_1, \gamma_2$ cobound, we see that $lk(\gamma_2,\gamma_1^+) = lk(\gamma_1,\gamma_1^+)$.  However, $\gamma_1^+$ is homologous to $\gamma_2^+$  in the complement of $\gamma_1 \cup \gamma_2$, and hence $lk(\gamma_2,\gamma_1^+) = lk(\gamma_2,\gamma_2^+)$, giving the desired claim.  

Since $\gamma_1$ and $\gamma_2$ cobound an embedded surface, they must be disjoint. Since Dehn twists corresponding to disjoint curves commute, we see that $T_{\gamma}^n \coloneqq (T_{\gamma_1}T_{\gamma_2}^{-1})^n= T_{\gamma_1}^nT_{\gamma_2}^{-n}$. Therefore Lemma \ref{lem:twist-surgery} completes the proof.    
\end{proof}

\begin{ex}\label{ex:knottotwist}
We can view $1/n$-surgery on any knot $K$ in a homology sphere $Y$ as Torelli surgery by powers of a separating twist on some embedded surface $S$; simply take a neighborhood in $Y$ of a Seifert surface for $K$. The boundary of this neighborhood will be an embedded surface in $Y$ with $K$ lying on it as a separating curve. Similarly, any surgery $Y_{s-\frac{1}{n},s+\frac{1}{n}}(K \cup J)$ on a two component link where $s$ is the surface framing for the connected Seifert surface of $K \cup J$ can be viewed as Torelli surgery by a bounding pair map by doubling the connected Seifert surface.
\end{ex}


In general, for $\phi \in \mathcal{I}_g$, we can describe $Y_{S, \phi}$ as surgery on a link in $Y$ using a factorization of $\phi$ into several bounding pair maps and separating Dehn twists; our goal is to understand this surgery description.  
If we factor $\phi$ as a product of separating twists and bounding pair maps $\phi_1^{n_1} \cdots \phi_k^{n_k}$ along $\gamma_1,\ldots,\gamma_k$, then we can obtain a link surgery description from $Y$ to $Y_\phi$ as follows: 
\begin{itemize}
\item Identify a neighborhood of $S$ in $Y$ with $S \times I$.  
\item The $i$th component (or pair of components) of this link is a copy of $\gamma_i$ on $S \times \{1/i\}$.  If it is a bounding pair map, the surgery coefficient is $s_i - 1/n_i$ on $\gamma_i^1$ and $s_i + 1/n_i$ on $\gamma_i^2$, where $s_i$ is the surface framing.  If it is a separating twist, then the surgery coefficient is $s_i - 1/n_i$.   
\end{itemize}

There is naturally a question as to whether the integer $s_i$ we associate to the surface framing of $S$ depends on whether we have done surgery on the other $\gamma_j$ first.  This is a somewhat subtle point and is addressed in the following lemma. 


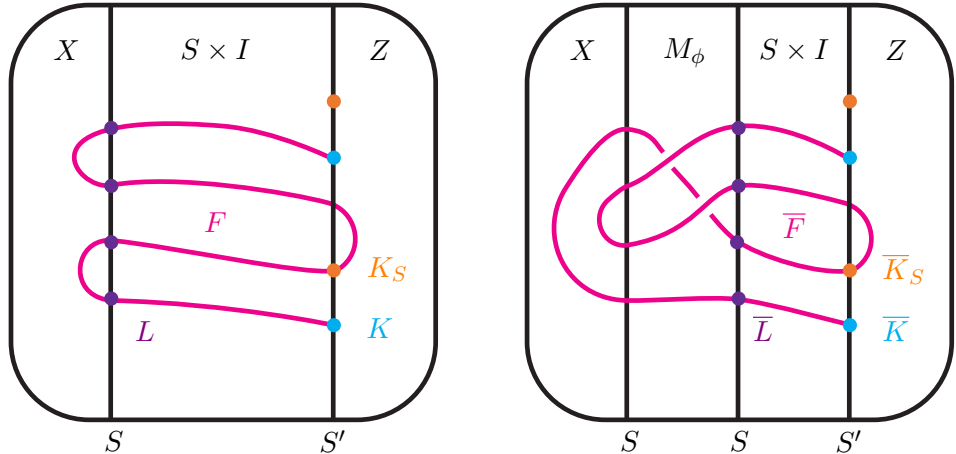
\begin{figure}[h] 
    \centering
      \begin{subfigure}{0.45\textwidth}
         \labellist                             
    \large\hair 2pt
    \pinlabel $S$ at 75 0
        \pinlabel {{\color{violet}$L$}} at 95 70
     \pinlabel $S^{\prime}$ at 215 0
    \pinlabel {{\color{orange}$K_S$}}  at 250 110
      \pinlabel {{\color{cyan}$K$}}  at 245 72
       \pinlabel {{\color{magenta}$F$}}  at 140 140
    \pinlabel $X$ at 45 250
      \pinlabel {$S \times I$}  at 140 250
          \pinlabel {$Z$}  at 245 250
    \endlabellist  
        \centering
      \includegraphics[width=6cm]{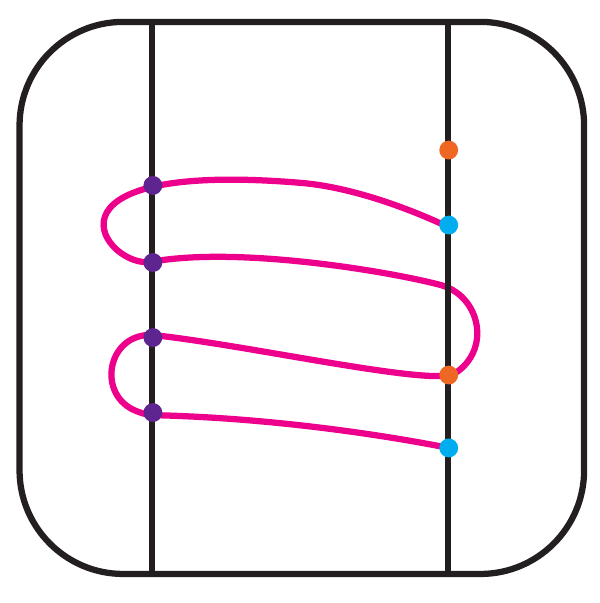}
    \end{subfigure}
     \begin{subfigure}{0.45\textwidth}
      \labellist                             
    \large\hair 2pt
    \pinlabel $S$ at 75 0
        \pinlabel {{\color{violet}$\overline{L}$}} at 160 72
    \pinlabel $S$ at 145 0
     \pinlabel $S^{\prime}$ at 215 0
    \pinlabel {{\color{orange}$\overline{K}_S$}}  at 250 110
      \pinlabel {{\color{cyan}$\overline{K}$}}  at 245 72
      \pinlabel {{\color{magenta}$\overline{F}$}}  at 180 138
    \pinlabel $X$ at 45 250
    \pinlabel $M_{\phi}$ at 110 248
      \pinlabel {$S \times I$}  at 180 250
          \pinlabel {$Z$}  at 245 250
    \endlabellist  
        \centering
           \includegraphics[width=6cm]{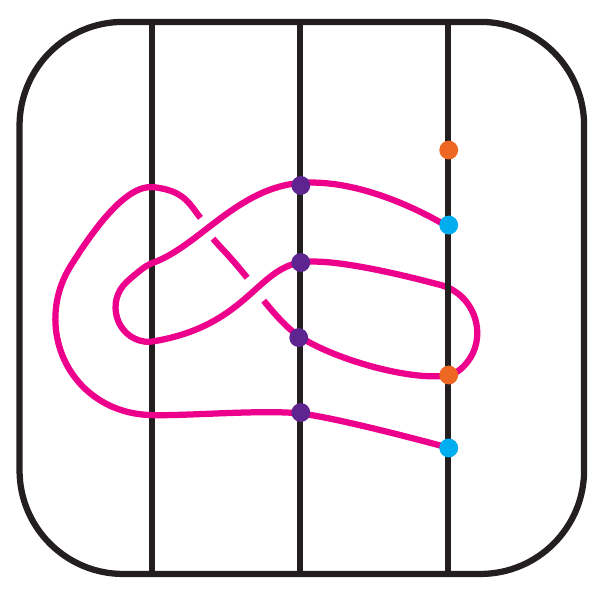}
    \end{subfigure}
    \caption{A schematic picture one dimension lower describing the proof of Lemma \ref{lem:framinglemma}.}
    \label{fig:framinglemma}
\end{figure}

\begin{lem}\label{lem:framinglemma}
Fix a surface $S$ in a homology sphere $Y$ and a Torelli element $\phi \in \mathcal{I}_g$.  Let $K$ be a simple closed curve on a parallel copy of $S$.  Then, the difference between the Seifert and surface framings of $K$ is unchanged by Torelli surgery corresponding to $\phi$.
\end{lem}
\begin{proof}
Let $K_S$ denote a surface-framed pushoff of $K$ in $Y$.  Let $\overline{K}$ denote the image of $K$ after Torelli surgery and $\overline{K}_S$ the surface-framed pushoff of $\overline{K}$, which is the same as the image of $K_S$ after Torelli surgery.  We would like to show that 
\[
F \cdot K_S = \overline{F} \cdot \overline{K}_S,
\]
where $F$ is a Seifert surface for $K$ and $\overline{F}$ is a Seifert surface for $\overline{K}$.  

Fix a Seifert surface  $F$ for $K$.  For notation, let $S'$ be the parallel copy of $S$ and write $$Y = X \ \cup_{S} \ S \times I \ \cup_{S'} Z$$ as in the left hand side of Figure \ref{fig:framinglemma}.  Let $L = F \cap S$.  Note that $L$ is nullhomologous in $X$, since it bounds $F \cap X$.  For the sake of this proof, we can understand $Y_{S,\phi}$ using the schematic in Figure \ref{fig:framinglemma} where we have constructed $Y_{S,\phi}$ by cutting open $Y$ along $S$ and gluing in the mapping cylinder $M_{\phi}$. Because $\phi \in \mathcal{I}_g$, $\overline{L}$ is still nullhomologous in $X  \ \cup_{S} \ M_{\phi}$.  Gluing the embedded surface representing this nullhomology to  $F \cap (S \times I \ \cup_{S'} Z ) $ inside $Y_{S,\phi}$ yields a Seifert surface $\overline{F}$ for $\overline{K}$ in $Y_{S,\phi}$. By construction, the intersection number between $F$ and $K_S$ is the same as that of $\overline{F}$ and $\overline{K}_S$.       
\end{proof}


\begin{rmk}
Torelli surgery produces an important relationship between the Johnson kernel and boundary links (links whose components bound disjoint Seifert surfaces).  More precisely, if $\phi \in \mathcal{K}_g$ is factored as a product of separating twists, then this factorization gives an expression of $Y_\phi$ as surgery on a boundary link in $Y$.      
\end{rmk}

\subsection{3-manifold invariants}
In this work we will focus on two 3-manifold invariants. The first, the Casson invariant, has multiple interpretations. The original definition is loosely a suitable count of conjugacy classes of irreducible $SU(2)$-representations of the fundamental group of the 3-manifold, and can be computed in terms of the $SU(2)$-representations of a Heegaard surface. Another is loosely a count of gauge equivalence classes of irreducible flat  $SU(2)$-connections \cite{taubes}. In our case, we will be interested in the characterization with the following axiomatic definition.

\begin{dfn}[See for example Theorem 12.1 in \cite{saveliev}]\label{def:casson} Let $\mathcal{S}$ be the set of diffeomorphsim classes of oriented integral homology 3-spheres. The Casson invariant is the unique map $\lambda: \mathcal{S} \rightarrow \mathbb{Z}$ satisfying
\begin{enumerate}
    \item $\lambda(S^3) = 0$; 
    \item For a knot $K$ in a homology sphere $Y$, 
    \begin{equation}\label{eq:casson-surgery}
        \lambda(Y_1(K)) - \lambda(Y) = \frac{1}{2} \Delta''_{K}(1),
    \end{equation} 
    where $\Delta_K(t)$ is the symmetrized Alexander polynomial of $K$ normalized so that $\Delta_K(0) = 1$; 
    \item $\lambda (S^3_{1}(T_{2,3}))=1$.
\end{enumerate}
\end{dfn}


Next, we will need the $d$-invariant.  While we will not need the definition, and instead mostly apply known results, we include it for context.  For any integral homology sphere $Y$, the Heegaard Floer homology  $HF^{-}(Y)$ is a graded $\Z/2\Z[U]$-module with the structure
\[ HF^{-}(Y) = \Z/2\Z[U]_{d(Y)} \oplus HF^{-}_{red}(Y)\]
where $d(Y)$ is the grading of $1$ in the free summand and $HF^{-}_{red}(Y)$ is the $U$-torsion submodule ($HF^-_{red}(Y)$ is called the {\em reduced} Heegaard Floer homology of $Y$).  This number $d(Y)$ is the $d$-invariant, which is always an even integer and is a homology cobordism invariant \cite{ozszdinvt}; more generally, $d(Y)$ is loosely governed by the four-manifolds bounding $Y$.  For more exposition on the $d$-invariants, see for example \cite{homsurvey}.

\section{Proofs of asymptotic results}\label{sec:proofs}

In this section, we begin by bounding the change in $d$-invariants under Torelli surgery.  Suppose that $\phi$ is a separating twist or bounding pair map on $\Sigma_g$.  By cutting along the corresponding curves in $\Sigma_g$, we obtain two surfaces.  In what follows, let $m$ denote the minimal value among the two genera.  
\begin{prop}\label{prop:onetorellielt}
Fix a homology sphere $Y$.  Let $\phi$ be an element of the Torelli group of genus $g$ which is either a separating twist or bounding pair map.  Then, 
    \[
    |d(Y) - d(Y_\phi)| \leq 2\lceil{\frac{m}{2}}\rceil.
    \]
\end{prop}
\noindent See also Remark 3.10 in \cite{LMPC}.  
\begin{proof}
This can be deduced from  \cite{LMPC} and known computations of the $d$-invariants of circle bundles over surfaces.  In the former, a linear bound on $|d(Y) - d(Y_\phi)|$ in terms of $m$ is established but not quantified.  We do not review the arguments in \cite{LMPC} and only describe the additional steps needed here to quantify the bounds explicitly.  First, we consider the case of a separating twist.  Since $d$-invariants change sign under orientation reversal, we may assume that $\phi$ is a positive Dehn twist, so $Y_\phi$ is $-1$-surgery on a knot of genus at most $m$ in $Y$.  In the proof of Proposition 3.6 in \cite{LMPC}, the following inequality is established:
\[
0 \leq d(Y_\phi) - d(Y) \leq m + d_b(N_{-1,m}), 
\]
where $d_b(N_{n,m})$ denotes the ``bottom'' $d$-invariant and $N_{n,m}$ denotes the circle bundle over Euler number $n$ over a genus $m$ surface.  For non-zero Euler numbers, this $d$-invariant was computed in \cite[Theorem 4.2.3]{Park-thesis}.  This $d$-invariant for $N_{-1,m}$ is particularly simple and is 0 if $m$ is even and 1 if $m$ is odd.  Therefore, we obtain
\[
|d(Y_\phi) - d(Y)| \leq \lceil{\frac{m}{2}}\rceil.
\]
(Note that this inequality is all that is necessary to prove Corollary~\ref{cor:surgery-bound}.)  

In the case of a bounding pair map, the proof of Theorem 1.9 in \cite{LMPC} establishes the inequality
\[
d(Y_\phi) - d(Y) + \underline{d}(N_{0,m}) \geq -\frac{2m+1}{2}, 
\]
where $\underline{d}$ denotes the twisted $d$-invariant of Behrens-Golla \cite{BehrensGolla}.  Behrens-Golla compute this twisted $d$-invariant to be $(-1)^{m+1}/2$ (\cite[Theorem 6.1]{BehrensGolla}), and so we obtain
\[
d(Y_\phi) - d(Y) \geq (-1)^m/2 - \frac{2m+1}{2} = -2\lceil{\frac{m}{2}}\rceil
\]
Note that $Y = (Y_\phi)_{\phi^{-1}}$, and $\phi^{-1}$ is a bounding pair map as well.  Therefore, we also obtain 
\[
d(Y) - d(Y_\phi) \geq -2\lceil{\frac{m}{2}}\rceil,
\]
which completes the proof.  
\end{proof}

We can now give a more general version of Theorem \ref{thm:lineargrowth} and see that the statement in the introduction is a special case. 


\begin{thm}\label{thm:preciselineargrowth} Let $S$ be an embedded surface of genus $g$ in $Y$ and $A$ be any set of separating twists and bounding pair maps.  Let $m_A = \max_{\phi \in A} m_\phi$ (where  $m_\phi$ is the constant $m$ for map $\phi$ in Proposition \ref{prop:onetorellielt}).  Let  $\phi \in  \langle A  \rangle$ be a gluing map where $\langle A  \rangle$ denotes the subgroup generated by $A$. If $Y = S^3_\phi$,  then
\[ |d(Y)| \leq 2\lceil \frac{m_A}{2} \rceil ||\phi||_A\]
where $||\phi||_A$ is the word length of $\phi$ in $ \langle A \rangle $. Since $\mathcal{I}_g$ is generated by bounding pair maps and separating curves, we can choose $A$ so that $\langle A  \rangle = \mathcal{I}_g.$  Further, a generating set $A$ for $\mathcal{I}_g$ can be chosen such that $m_A = 1$.  If $g >2 $, we can choose a generating set $A$ which is finite.  
\end{thm}

\begin{proof}  The desired inequality follows from a direct application of Proposition \ref{prop:onetorellielt}.  The case of $m_A = 1$ follows by work of \cite{johnson1985structure}, which shows that a generating set can always be chosen to consist of twists on curves which split off genus 1 subsurfaces from $\Sigma_g$. 
By work in \cite{johnson1985structure} and \cite{mcculloughmiller}, $\mathcal{I}_g$ is finitely generated exactly when $g>2$. 
\end{proof}

It is natural to ask whether these results are really special to the structure of the Torelli group as opposed to the mapping class group as a whole.  Of course, one can frequently apply surgery using elements not in Torelli and get back a homology sphere. One should not a priori expect these results to apply to the mapping class group as whole, particularly because the word metric in the Torelli group is distorted at least exponentially when included into the mapping class group (and in fact, even into the mod p congruence subgroups) \cite{bfpdistortion} \cite{kunoomori}. In fact, the following lemma shows we cannot bound the $d$-invariant of a homology sphere obtained by even a single Dehn twist along an arbitrary curve. 

\begin{figure}[ht]
    \centering
        \includegraphics[width=0.4\linewidth]{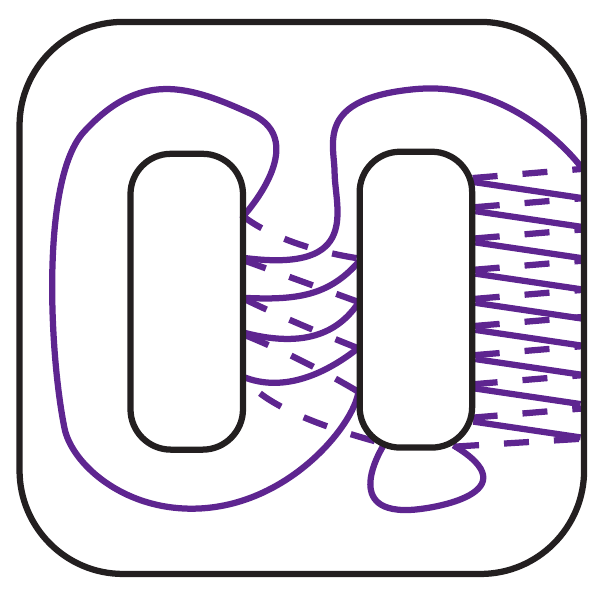}
    \caption{An example of a curve $\gamma_k$ in the proof of Lemma \ref{lem:unbounded} with $k=1$.}
    \label{fig:unbounded}
\end{figure}

\begin{lem} \label{lem:unbounded}
Fix a Heegaard surface $S$ for a homology sphere $Y$ of genus at least 2.  There does not exist a constant $C$ such that
\[
|d(Y_{S,\tau})| \leq C
\]
for every Dehn twist $\tau$ on $S$ producing a homology sphere.
\end{lem}
\begin{proof}
By Theorem~\ref{thm:lineargrowth}, it suffices to prove the result for the standard genus $g$ Heegaard surface $S$ for $S^3$.  Place the torus knot $T_{2,4k+1}$ on a genus 1 subsurface so that it corresponds to a $(2,4k+1)$-curve should that genus 1 surface be capped off.  In this picture, the surface framing is $8k+2$. Therefore, modify the torus knot so that an arc runs along the longitude of the complementary subsurface of $S$ and wraps meridionally $-(8k+2)$ times as in Figure \ref{fig:unbounded} for the case $k=1$. The resulting surface framing from $S$ is now zero.  Let $\tau$ denote the negative Dehn twist along this new curve which we will call $\gamma_k$, which is clearly not in Torelli (the curve is homologically essential in $S$).  Because the surface framing is 0, $S^3_\tau = S^3_1(T_{2,4k+1})$, which is a homology sphere with $d = -2k$.  The result follows.  
\end{proof}



For comparison, we also quickly observe that the quadratic bounds on the Casson invariant of Broaddus-Farb-Putman do not apply for the word metric using the standard infinite generating sets for the Torelli group.
\cassonfail*
\begin{proof}
Let $K_n$ denote the twist knot $(2n+1)_1$.  Consider the orientable checkerboard surface for the standard projection of $K_2$; a neighborhood of this surface provides half of a Heegaard surface for $S^3$ where $K_2$ sits as a separating curve as seen in Figure \ref{fig:twistknotonsfce}.  Further, all other $n$ are obtained by inserting meridional twists along the curve in the right side of the figure.  Note that all the $K_n$ still sit on this surface as separating curves.  Consequently, the manifolds $S^3_1(K_n)$ are all obtained from $S^3$ by Torelli surgery using a single separating twist, i.e. an element with word length 1.  However, $|\lambda(S^3_1(K_n))| = n$ by the surgery formula for the Casson invariant, as $K_n$ has Alexander polynomial $nt - (2n-1) + nt^{-1}$ .  Hence there is no attainable bound on the change in the Casson invariant under Torelli surgery in terms of $\| \cdot \|_{\infty}$.  
\end{proof}


\begin{figure}[ht]
    \centering
     
     \begin{subfigure}{0.45\textwidth}
        \centering
        \includegraphics[width=0.9\linewidth]{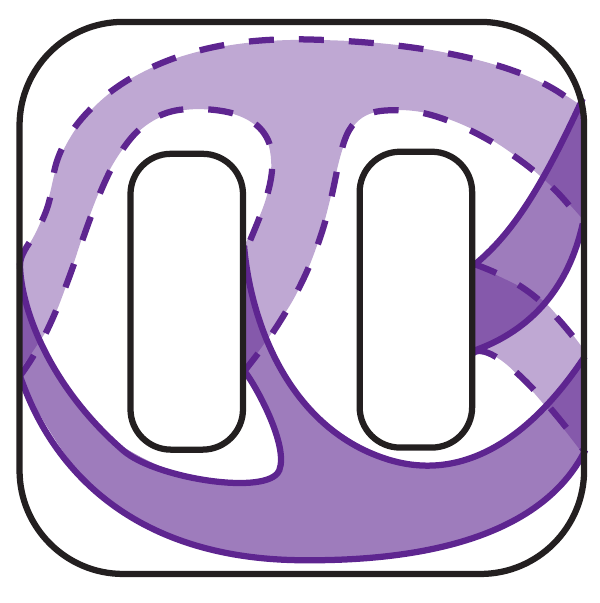}
    \end{subfigure}
    \begin{subfigure}{0.45\textwidth}
        \centering
        \includegraphics[width=0.9\linewidth]{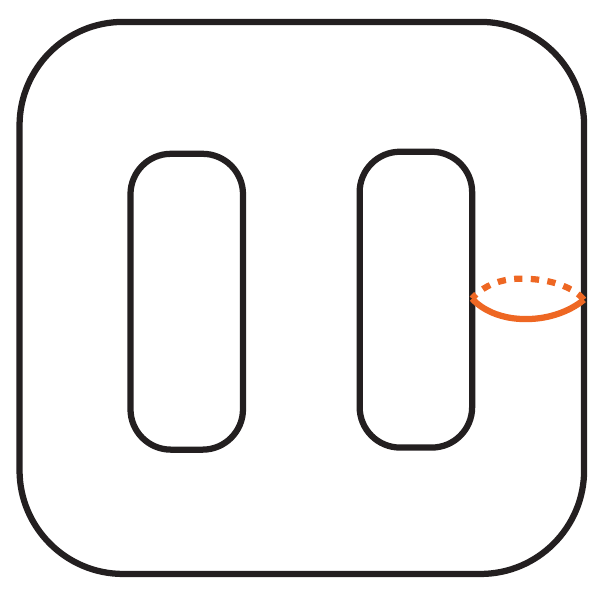}
    \end{subfigure}
    \caption{On the left is the twist knot $K_1$ as a separating curve, together with the subsurface it bounds on $\Sigma_2$. If we conjugate the Dehn twist $T_{K_1}$ by Dehn twists on the meridional curve on the right, we obtain the Dehn twist $T_{K_n}$ which is also a separating twist.}
    \label{fig:twistknotonsfce}
\end{figure}

To further illustrate the difference between the behavior of the Casson invariant and the $d$-invariant, we prove the following.

\finitetype*
\begin{proof}
We will show the $d$-invariant is not a finite type invariant.  Since the Euler characteristic of $HF_{red}$ is a linear combination of the $d$-invariant and the Casson invariant, which is a finite type invariant, the result for the Euler characteristic follows.

If $I$ is a finite-type invariant of homology three-spheres of order at most $n-1$, then $I(S^3_\Lambda(L)) = I(S^3)$ for any Brunnian link with $n$ components by definition \cite{ohtsuki}; here $\Lambda$ denotes a choice of $\pm 1$ for each component of $L$.  Therefore, it suffices to construct Brunnian links with an arbitrarily large number of components such that the $d$-invariant of some integral surgery is non-zero.  Begin with the Borromean rings and take the $(n-3)$-fold Bing double of one of the components. Call the resulting Brunnian $n$-component link $L_n$ (see Figure \ref{fig:iteratedBR}).  We claim that if $\Lambda = (1,\ldots, 1)$, then $d(S^3_\Lambda(L_n)) \neq 0$.  To see this, we first will show that $S^3_\Lambda(L_n) = S^3_1(Wh^{(n-3)}(T_{2,3}))$, where $Wh^{(k)}$ denotes the $k$-fold iteration of the positive Whitehead double satellite operation.  Indeed, let $K$ be a knot and $J$ the result of Bing doubling $K$.  Then, $(+1,+1)$-surgery on the Bing double can easily be seen to be $+1$-surgery on $K$ by blowing down one of the components of the Bing double.  We can view $(+1,\ldots,+1)$-surgery on $L_n$ as $(+1,\ldots,+1)$ on the $(n-3)$-fold Bing double of the right-handed trefoil, since the operations of Bing doubling and blowing down the two other components of the Borromean rings commute.  (This is because the linking numbers of the Borromean rings are zero.)  Hence, we see that $S^3_\Lambda(L_n) = S^3_1(Wh^{(n-3)}(T_{2,3}))$.

Hedden shows that $\tau(Wh^{(k)}(T_{2,3})) = 1$ for any $k$ \cite[Theorem 1.4]{HeddenWhitehead}, and so  $d(S^3_\Lambda(L_n)) \neq 0$ (see \cite[Equation (1.1)]{HomWu} and \cite[Proposition 1.6]{NiWu}, for example), completing the proof.
\end{proof}

\begin{figure}[ht]
    \centering
        \includegraphics[width=0.4\linewidth]{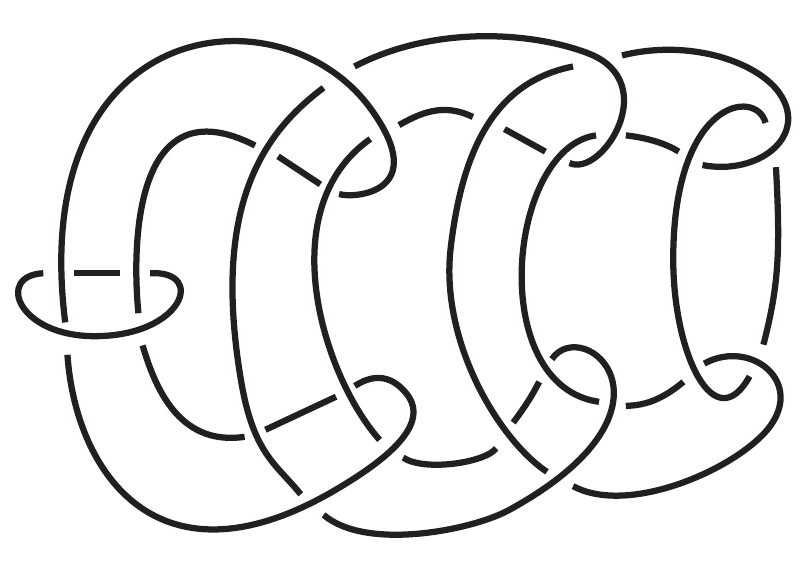}
    \caption{An example of a link $L_n$ in the proof of Theorem \ref{thm:finitetype} with $n=5$.}
    \label{fig:iteratedBR}
\end{figure}

\section{Families of homology spheres with bounded d-invariants}\label{sec:bounded}

With the goal of proving Theorem~\ref{thm:family}, we now show that many families of Torelli surgeries have a bounded effect on the $d$-invariants of homology spheres.  First, we study separating Dehn twists, then bounding pair maps, and then we allow for combinations of the two.

\subsection{Separating Dehn twists}
We begin with a technical Dehn surgery lemma that allows us to compare the $d$-invariants of surgery on a link in an arbitrary homology sphere to those of surgery on a link in $S^3$.

\begin{lem}\label{lem:surgery-convert}
Let $L = L_1 \cup \ldots \cup L_k$ be a link in an integer homology sphere $Y$.  Choose another link $L' = L'_1 \cup \ldots \cup L'_k$ in $S^3$ with the same pairwise linking numbers as $L$.  Then there is a constant $C$ such that,  
\[
|d(Y_{\vec{r}}(L)) - d(S^3_{\vec{r}}(L'))| \leq C,
\]
for any $\vec{r} = (r_1,\ldots,r_k) \in \mathbb{Q}^k$ such that $Y_{\vec{r}}(L)$ is an integer homology sphere.  
\end{lem}
\begin{proof}
We carry this out in two steps.  In the first step, we will transform $(Y,L)$ into $(S^3,L'')$ for {\em some} link $L''$ with the same linking numbers as $L$ and show there is a constant $K_1 \in \mathbb{R}_{\geq 0}$ satisfying 
\[
|d(Y_{\vec{r}}(L)) - d(S^3_{\vec{r}}(L''))| \leq K_1,
\]
for any $\vec{r} = (r_1,\ldots,r_k) \in \mathbb{Q}^k$ such that $Y_{\vec{r}}(L)$ is an integer homology sphere.  Then, we relate surgery on $L''$ to surgery on $L'$, for an arbitrary $L'$ with the same linking numbers.  

By \cite[Proposition 2.4, Remark 2.6]{LMPC}, there is a link $L''$ in $S^3$ and a sequence of surgeries transforming $(S^3, L'')$ into $(Y,L)$; furthermore, for each $i$, the $i$th component in this sequence of surgeries is nullhomologous in the complement of $L''$ after the previous $i-1$ surgeries, and all the surgeries have coefficients $\pm 1$.  In particular, there is a sequence of $\pm 1$-surgeries on nullhomologous knots transforming $S^3_{\vec{r}}(L'')$ into $Y_{\vec{r}}(L)$ whose genera at each step is bounded independent of $\vec{r}$.  By Corollary \ref{cor:surgery-bound} (or \cite[Proposition 3.6]{LMPC}), we see that there is a constant $K_1 \in \mathbb{R}_{\geq 0}$ such that 
\begin{equation}\label{eq:YtoS3}
|d(Y_{\vec{r}}(L)) - d(S^3_{\vec{r}}(L''))| \leq K_1,
\end{equation}
for any $\vec{r} = (r_1,\ldots,r_k) \in \mathbb{Q}^k$ such that $Y_{\vec{r}}(L)$ is an integer homology sphere.  

Now, we convert $L''$ to $L'$ and study the change in $d$-invariants.  By a sequence of Torelli surgeries in the complement of $L''$, we can transform $(S^3,L'')$ into $(S^3,L')$, where $L'$ is any link with the same pairwise linking numbers as $L''$ (i.e. the same as $L$).  See \cite[Proof of Lemma 2]{Matveev} or \cite{MurakamiNakanishi}.  Each Torelli surgery may occur along a different handlebody of possibly differing genus.  Nonetheless, we can transform $S^3_{\vec{r}}(L'')$ into $S^3_{\vec{r}}(L')$ using a sequence of Torelli surgeries such that the length of the sequence, the genera of the relevant surfaces, and the number of separating Dehn twists and/or bounding pair maps to describe them are all independent of $\vec{r}$.  Theorem~\ref{thm:lineargrowth} (or \cite[Theorem 1.9]{LMPC}) implies there is an integer $K_2$ such that 
\begin{equation}\label{eq:S3toS3}
|d(S^3_{\vec{r}}(L'')) - d(S^3_{\vec{r}}(L'))| \leq K_2,
\end{equation}
for any $\vec{r} = (r_1,\ldots,r_k) \in \mathbb{Q}^k$ such that $S^3_{\vec{r}}(L'')$ is an integer homology sphere.  Of course, this is the same as the $\vec{r}$ which will make $Y_{\vec{r}}(L)$ an integer homology sphere, so let $C=K_1+K_2$ and the result follows from applying the triangle inequality to \eqref{eq:YtoS3} and \eqref{eq:S3toS3}.  
\end{proof}

As a warm-up towards proving Theorem~\ref{thm:family}, we begin by showing that iterating a single separating twist produces only a bounded change in $d$-invariants.  
\begin{prop}
Let $Y$ be an integer homology sphere and $S$ an embedded surface in $Y$ with genus $g$.  Let $\phi$ be a separating twist.  Then, $\{d(Y_{S,\phi^n})\}_{n \in \Z}$ is bounded.  
\end{prop}
\begin{proof}
Since $\phi$ is a separating Dehn twist, $Y_{\phi^n}$ can be described as $-1/n$-surgery on a knot $K$ in $Y$ by Lemma \ref{lem:twist-surgery}.  By Lemma~\ref{lem:surgery-convert}, for any knot $J$ in $S^3$, the quantity $|d(Y_{S,\phi^n}) - d(S^3_{-1/n}(J))|$ is bounded independent of $n$.  Therefore, choose $J$ to be unknotted, and $S^3_{-1/n}(J) = S^3$ for all $n$.  This completes the proof.  
\end{proof}

We can quickly prove a stronger version, which is Theorem~\ref{thm:family} if one restricts only to separating twists.  
\begin{prop}\label{prop:separating}
Let $Y$ be an integer homology sphere and $S$ any genus $g$ surface embedded in $Y$.  Let $\phi_1,\ldots,\phi_k$ be separating Dehn twists on $S$ and let $\Gamma$ be the set of the elements consisting of $\phi_1^{n_1} \cdots \phi_k^{n_k}$ with $n_1,\ldots,n_k \in \Z$.  Then, the collection of $|d(Y_\phi)|$ indexed over $\phi \in \Gamma$ is bounded.
\end{prop}
\begin{proof}
Identify a neighborhood of $S$ in $Y$ with $S \times [0,1]$.  Let $\phi_i$ denote a Dehn twist along $L_i \subset S \times \{1/i\}$.  Write $L = L_1 \cup \ldots \cup L_k$, a link in $Y$.  Since each $L_i$ is separating on $S\times \{1/i\}$ and $L_j$ is disjoint from $S \times \{1/i\}$ for $i\neq j$, $L_i$ bounds a surface disjoint from $L_j$, and we see that all pairwise linking numbers in $L$ are $0$.

From the above description, it follows that if $\phi = \phi_1^{n_1} \circ \ldots \circ \phi_k^{n_k}$, then $Y_\phi$ is obtained by performing $(-\frac{1}{n_1},\ldots,-\frac{1}{n_k})$-surgery on $L$.  We would like to see that the set of $d$-invariants of this family of manifolds is bounded.    By Lemma~\ref{lem:surgery-convert}, we can replace $Y_\phi$ with $(-\frac{1}{n_1},\ldots,-\frac{1}{n_k})$-surgery on any link in $S^3$ with pairwise linking zero.  Choose the unlink and the result follows.  
\end{proof}

\subsection{Bounding pair maps}
Next, we consider the case of words of bounding pair maps similar to Proposition~\ref{prop:separating} and show these $d$-invariants are also bounded.  Our strategy will be similar to that used in Proposition~\ref{prop:separating}.  We begin with some technical preliminaries.  


\begin{lem}\label{lem:cable}
Let $L = L_1 \cup L_2$ be an oriented two-component link in a homology sphere $Y$ and let $L'$ be another framed link disjoint from $L$.  Suppose that $L_1, L_2$ cobound an annulus in the exterior of $L'$ and have linking number $-\ell$ with this orientation.  Consider meridians $\mu_i$ of $L_i$.   Perform $\ell$-surgery on $L_1$ and on $L_2$, $n$-surgery on the meridian of $L_1$ and $-n$-surgery on the meridian of $L_2$.  The result is $Y$ and $L'$ is transformed into a link of the same framed isotopy type.    
\end{lem}
\begin{proof}
Let $A$ be the annulus that $L$ cobounds.  Taking a neighborhood of $A$ in $Y$, we see that $L_1$ and $L_2$ cobound annului on a torus $T$ in $Y$.  By Lemma~\ref{lem:boundingpair-surgery}, we see that the given surgery on this link corresponds to removing a neighborhood of $T$ and regluing by a bounding pair map along the curves corresponding to $L_1, L_2$ taken to the $n$th power.  This bounding pair map is isotopic to the identity mapping class on $T$, and so the surgery operation preserves $Y$ and $L'$.   
\end{proof}

\begin{rmk}
This can also be proved directly using Kirby calculus.
\end{rmk}

We will now use this lemma to prove the following. 

\begin{prop}\label{prop:bounding-pair}
Let $\phi_1,\ldots,\phi_k$ be bounding pair maps on an embedded surface $S$ in a homology sphere $Y$.  Then, the set $\{d(Y_{S, \phi_1^{n_1} \cdots \phi_k^{n_k}})\}$ is bounded.  
\end{prop}
\begin{proof}
The strategy is to write $Y_{\phi_1^{n_1} \cdots \phi_k^{n_k}}$ as some surgery on a fixed link in $Y$ (independent of $n_1,\ldots,n_k$) and then use Torelli surgeries (again independent of $n_1,\ldots, n_k$) to transform this into a link in $S^3$ whose surgery we can compute.  To streamline notation, let $Y(n_1,\ldots,n_k)$ denote $Y_{\phi_1^{n_1} \cdots \phi_k^{n_k}}$.  

As in the case of separating twists, identify a neighborhood of $S$ with $S \times I$.  For each $1 \leq i \leq k$, let $\phi_i$ be represented by $T_{\alpha_i} T^{-1}_{\beta_i}$ where $T_\gamma$ denotes a Dehn twist along $\gamma$ and $\alpha_i, \beta_i$ are curves on $S \times \{1/i\}$.  We may view $\alpha_i, \beta_i$ as a 2-component link $L_i$ in $Y$.  Write the pairwise linking as $\ell_i$, where $\alpha_i$ and $\beta_i$ are oriented to be homologous in $S$.  Consider the 4-component link $M_i$ which is obtained by adding meridians to both $\alpha_i$ and $\beta_i$, called earrings.  By Lemma~\ref{lem:boundingpair-surgery} and the discussion corresponding to Figure \ref{fig:twist-surgery}, $Y_{S,\phi_i^{n_i}}$ is obtained by integral surgery on $M_i$ where each of $\alpha_i, \beta_i$ has framing $\ell_i$, the earring on $\alpha_i$ has framing $-n_i$ and the earring on $\beta_i$ is given framing $n_i$.  We will use these earrings to avoid rational surgery coefficients and therefore for the rest of this section we will refer to framed links rather than rational surgeries on links.

By construction, each pair $\alpha_i$ and $\beta_i$ is disjoint in $Y$.  Therefore, we can make sense of $\mathbb{L} = \bigcup_i L_i$ and $\mathbb{M} = \bigcup_i M_i$ as links in $Y$.  In this case, $Y(n_1,\ldots,n_k)$ can be expressed as surgery on $\mathbb{M}$ with the same surgery coefficients as we used for the individual $\phi_i^{n_i}$, i.e. $\ell_i$ on $\alpha_i, \beta_i$ and $-n_i, n_i$ on the meridians of $\alpha_i$ and $\beta_i$ respectively.  Therefore, we can realize all $Y(n_1,\ldots,n_k)$ as surgery on a single link, where the only dependence on the $n_i$ is in terms of the framings.  Now, following the strategy of the previous subsection, we would like to apply Torelli transformations to a link in $S^3$ we understand with the same pairwise linkings and then use it to bound the $d$-invariant of our original link.

We now construct a candidate link in $S^3$. First, because $\alpha_i$ and $\beta_i$ cobound a surface in  $S \times \{1/i\}$ which is disjoint from $S \times \{1/j\}$ for any $j \neq i$, $\alpha_i$ and $\beta_i$ are homologous in the complement of $\alpha_j$ in $Y$ for $j \neq i$.  Therefore, for $j \neq i$, $lk(\alpha_i, \alpha_j) = lk(\beta_i, \alpha_j)$.  Applying the same for $\beta_j$, and also switching the roles of $i$ and $j$, we see
\[
lk(\alpha_i, \alpha_j) = lk(\beta_i, \alpha_j) = lk(\alpha_i, \beta_j) = lk(\beta_i, \beta_j).  
\]
Call this linking number $\eta_{ij}$.  Note that the earrings of $\alpha_i$ and $\beta_i$ are geometrically unlinked from any component with index $j \neq i$.  

For any link $\mathbb{J}$ in $S^3$ with the same pairwise linking as $\mathbb{M}$, Lemma~\ref{lem:surgery-convert} tells us that
\begin{equation}\label{eq:surgery-d-convert}
|d(Y_\Lambda(\mathbb{M})) - d(S^3_\Lambda(\mathbb{J}))| \leq C, 
\end{equation}
where $C$ is a constant independent of the framing $\Lambda$ on $\mathbb{M}$.  

Let $\mathbb{K} = K_1 \cup \ldots \cup K_k$ be any link in $S^3$ with pairwise linking $\eta_{ij}$ for each component.  Let $\mathbb{J}$ be the link obtained by applying a $(2,2\ell_i)$ cable to $K_i$, which now doubles the number of components, and adding earrings to each of these $2k$ components.  See Figure~\ref{fig:cablingandearrings}.   It is easy to see that $\mathbb{J}$ has the same pairwise linking as $\mathbb{M}$.  If we let $Z(n_1,\ldots,n_k)$ denote the result of surgery on $\mathbb{J}$ with surgery coefficients corresponding to those of $\mathbb{M}$, then we see from \eqref{eq:surgery-d-convert} that 
\begin{equation}\label{eq:d-bounded}
|d(Y(n_1,\ldots, n_k) - d(Z(n_1,\ldots,n_k))| \leq C,
\end{equation}
for some constant $C$ independent of $n_1,\ldots, n_k$.  Hence, it suffices to show that $d(Z(n_1,\ldots, n_k))$ is bounded independent of $n_1,\ldots, n_k$ and then we can apply the triangle inequality as in the proof of Lemma \ref{lem:surgery-convert}.

\begin{figure}
    \centering
    \begin{subfigure}{0.3\textwidth}
        \labellist
\small\hair 2pt
\pinlabel $\alpha_1$ at 57 246
\pinlabel $\alpha_2$ at 270 225
\pinlabel $\beta_1$ at 76 26
\pinlabel $\beta_2$ at 190 255
\endlabellist
        \centering
        \includegraphics[width=\linewidth]{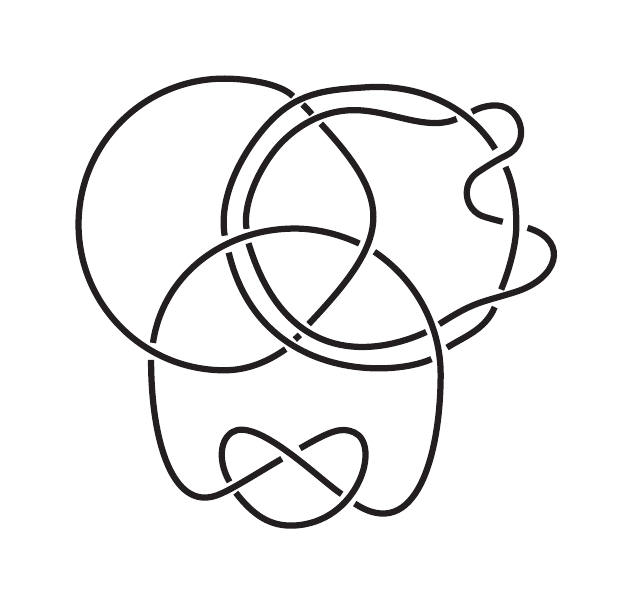}
               \caption{$\mathbb{L}$}
    \end{subfigure}
    \begin{subfigure}{0.32\textwidth}
        \labellist
\small\hair 2pt
\pinlabel $K_1$ at 35 110
\pinlabel $K_2$ at 270 110
        \endlabellist
        \centering
        \includegraphics[width=\linewidth]{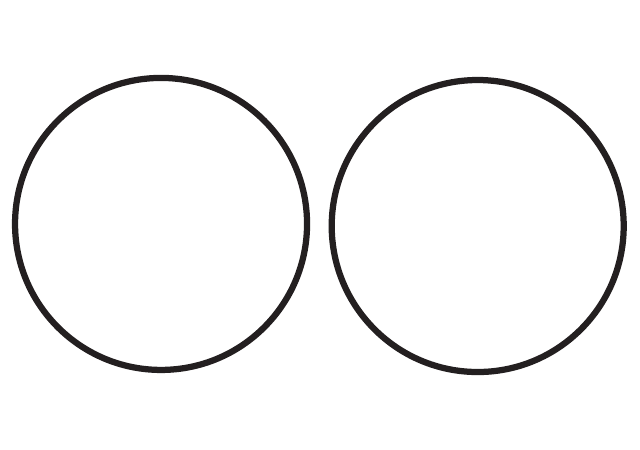}
        \caption{$\mathbb{K}$}
    \end{subfigure}
    \begin{subfigure}{0.32\textwidth}
                \labellist
            \small\hair 2pt

        \pinlabel $\mathbb{J}_1$ at 75 15 
        \pinlabel $\mathbb{J}_2$ at 230 15 
                \endlabellist
        \centering
        \includegraphics[width=\linewidth]{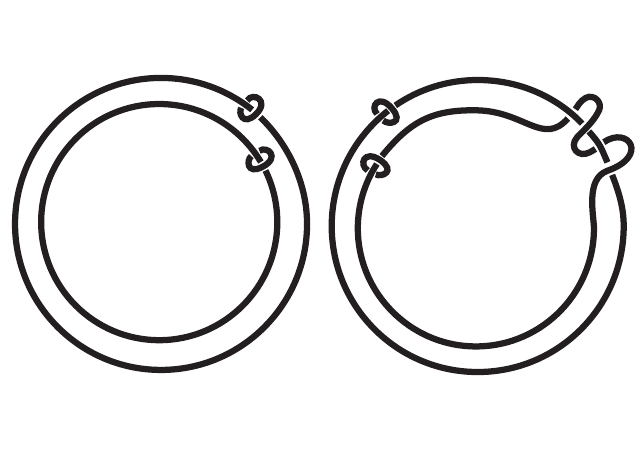}
        \caption{$\mathbb{J}$}
    \end{subfigure}
    \caption{Here $\mathbb{L}$ is a 4-component link representing curves associated to two bounding pair maps; the two furthest left components correspond to the first map $\phi_1$ and have $\ell_1=0$ and the remaining two components correspond to the second map $\phi_2$ and have $\ell_2=2$. We see that $\eta_{12}=\eta_{21}=0$. We now construct $\mathbb{K}$ as a $2$-component link with linking number $\eta_{12}=0$. Since this link is arbitrary, in the diagram we have the simplest link with this linking number. Adding in the meridian of each component will result in the 8-component link $\mathbb{M}$. We can now build $\mathbb{J}$ by cabling the components of $\mathbb{K}$ in order to have the same linking numbers as $\mathbb{L}$. We have taken the $(2,0)$ cable of $K_1$ and the $(2,4)$ cable of $K_2$ and included the meridians of each component to obtain $\mathbb{J}$. Now, by \cite{Matveev} a sequence of Torelli surgeries takes Dehn surgery on $\mathbb{M}$ with framings described in the proof of Proposition \ref{prop:bounding-pair} to Dehn surgery on $\mathbb{J}$ described in this same proof. This sequence is independent of the powers $n_1$ and $n_2$, so we obtain a single upper bound for the $d$-invariants of Torelli surgery by all words $\phi_1^{n_1}\phi_2^{n_2}$ (and similarly  $\phi_2^{n_2}\phi_1^{n_1}$).} 
    \label{fig:cablingandearrings}
\end{figure}

Fix $n_1,\ldots, n_k$.  For notation, let $\mathbb{J}_i$ denote the $(2,2\ell_i)$ cable of $K_i$ with meridians attached.  We will run the surgery in steps.  First, do $\ell_1$-surgery on the two components of the $(2,2\ell_1)$ cable of $K_1$, $n_1$-surgery on one meridian, and $-n_1$ surgery on the other meridian.  By Lemma~\ref{lem:cable} and the construction of $\mathbb{J}_1$, this surgery on $\mathbb{J}_1$ produces $S^3$ and preserves the framed link $\mathbb{J} - \mathbb{J}_1$.  By induction, we see that the corresponding framed surgery on $\mathbb{J}$ gives $S^3$.  This is of course independent of $n_1,\ldots, n_k$.  Hence $Z(n_1,\ldots,n_k) = S^3$ and by Equation~\eqref{eq:d-bounded}, it follows that $d(Y(n_1,\ldots,n_k))$ is uniformly bounded.  
 \end{proof}
 
\subsection{Combining separating twists and bounding pair maps}
We are able to combine the ideas from the proofs of Proposition~\ref{prop:separating} and Proposition~\ref{prop:bounding-pair} into one result.

\family*

\begin{proof}
Since there are only finitely many length $q$ sequences of elements in $\{1,\ldots,k\}$, it suffices to show that the set 
\[
\{d(Y_a) \mid a = \phi^{n_1}_{i_1} \ldots \phi^{n_q}_{i_q}, \ n_1,\ldots,n_q \in \mathbb{Z} \}
\]
is bounded for a fixed sequence $i_1,\ldots, i_q$.  

Without loss of generality, $\phi_{1},\ldots, \phi_{m}$ are separating twists and $\phi_{m+1}, \ldots, \phi_k$ are bounding pair maps. Identify a neighborhood of $S$ in $Y$ with $S \times I$. For $1 \leq j \leq q$, if $1 \leq i_j \leq m$, let $\gamma_j$ be the curve on $S \times \{1/j\}$ in $Y$ on which the separating Dehn twist is performed; if $m+1 \leq i_j \leq q$ let $\alpha_{j}, \beta_j$ on $S \times \{1/j\}$ be the pair of curves corresponding to $\phi_{i_j}$ as in the previous subsection.  By construction, the $\gamma$'s and all $\alpha,\beta$ are disjoint, and we can view the collection of all of these curves as a link $L$ in $Y$.  Upgrade this to a link $\mathbb{L}$ where we have attached meridians to the $\alpha$ and $\beta$ curves as in the previous subsection.  

By definition, $\gamma_{j}$ bounds a subsurface in $S \times \{1/j\}$, say $F_{j}$, which is disjoint from $\alpha_{j'}, \beta_{j'}$ for all $j \neq j'$, and so  
\begin{equation}\label{eq:linking}
lk(\gamma_{j}, \alpha_{j'}) = lk(\gamma_{j}, \beta_{j'}) = 0.  
\end{equation}
Of course, the only components that link the meridians are the curves which they are meridians of.  As in Proposition~\ref{prop:separating} and Proposition~\ref{prop:bounding-pair}, Lemma~\ref{lem:surgery-convert} implies it suffices to consider a link $\mathbb{T}$ in $S^3$ with the same pairwise linking as $\mathbb{L}$, and show that the corresponding surgeries on $\mathbb{T}$ have bounded $d$-invariants, independent of $n_1,\ldots,n_q$. 

Equation~\eqref{eq:linking} tells us that we can replace $\mathbb{T}$ with a split link in $S^3$, where, after re-ordering the components, the components corresponding to separating twists form an unlink and the remaining components corresponding to bounding pair maps form a link analogous to $\mathbb{J}$ from the proof of Proposition~\ref{prop:bounding-pair} (i.e. collections of $(2,\ell_i)$ cables with meridians).  Combining the proofs for the cases of separating twists and bounding pair maps shows that the corresponding surgeries on this link compute $S^3$.  Again, Lemma~\ref{lem:surgery-convert} now implies that the $d$-invariants of $Y_{\phi_1^{n_1} \cdots \phi_q^{n_q}}$ are bounded independent of $n_1,\ldots,n_q$.  This completes the proof.
\end{proof}

\section{Rational Homology Spheres}\label{sec:rational}
As outlined in the introduction, the relationships between invariants of integral homology spheres and the algebraic structure of $\mathcal{I}_g$ are quite surprising. For example, the Birman-Craggs-Johnson homomorphism \cite{birmancraggs, johnson} is constructed using the Rokhlin invariant. One of the reasons the Morita formula \eqref{eq:moritaformula} is so powerful is that it exactly measures the failure of the Casson invariant to be a homomorphism on $\mathcal{I}_g$ and is related to what Morita calls the ``core" of the Casson invariant. In \cite{moritacharclass}, he develops these ideas in the setting of characteristic classes of surface bundles. Therefore, if there were a similar framework to relate invariants of rational homology spheres and $\modg (k)$ one might hope to extend Morita's (and others') results in a straightforward way. However, we show this is not possible. 

\brokenmorita*

\begin{proof}[Proof of Theorem~\ref{thm:brokenmorita}]
Let $S$ denote a genus $g$ Heegaard surface $S$ for $S^3$.  We will construct separating twists $\psi, \eta$ and a non-separating Dehn twist $\phi$ such that for any nonzero $k \in \mathbb{Z}$
\[
\lambda(S^3_{\psi\phi^k}) = \lambda(S^3_{\psi}) + \lambda(S^3_{\phi^k}),
\]
and 
\[
\lambda(S^3_{\eta \phi^k}) = \lambda(S^3_{\eta}) + \lambda(S^3_{\phi^p}) + k,
\]
and also a self-homeomorphism of $S$ such that $\phi$ is preserved under conjugation and $\psi$ is sent to $\eta$.  All of the three-manifolds are homology spheres.  This will show that a Morita formula cannot exist defined purely in terms of the embedded Heegaard surface (although there may be a formula which depends on the embedding itself).    

\begin{figure}
    \centering
     
     \begin{subfigure}{0.45\textwidth}
        \centering
        \labellist                             
    \large\hair 2pt
        \pinlabel {{\color{violet}$K$}} at 300 72
    \pinlabel {{\color{orange}$C$}}  at 300 190
      \pinlabel {{\color{magenta}$M$}}  at 190 190
    \endlabellist  
        \includegraphics[width=0.9\linewidth]{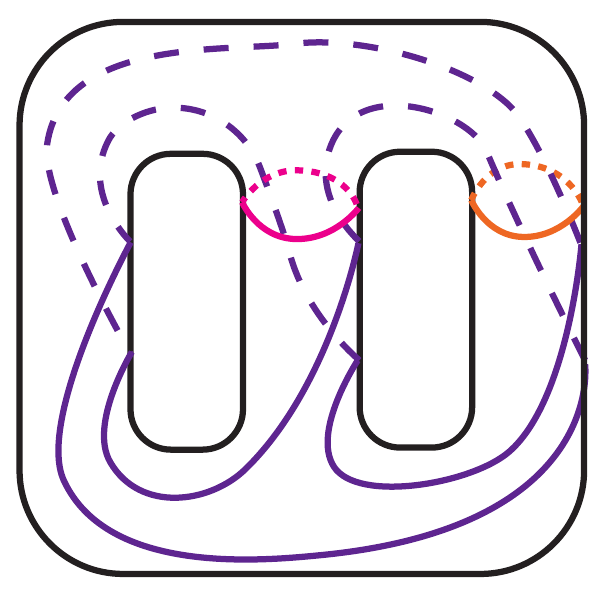}
    \end{subfigure}
    \begin{subfigure}{0.45\textwidth}
        \centering
         \labellist 
        \large\hair 2pt
        \pinlabel {{\color{teal}$L$}} at 300 72
    \endlabellist  
        \includegraphics[width=0.9\linewidth]{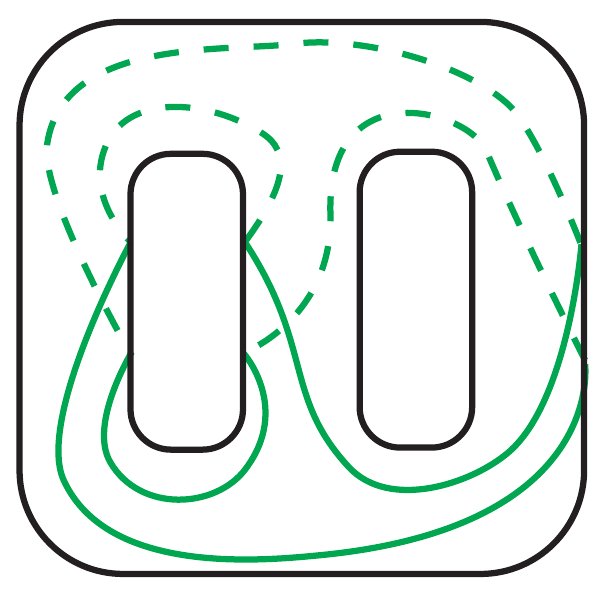}
    \end{subfigure}
    \caption{}
    \label{fig:rationalmorita}
\end{figure}

Consider the curves $C, K, L$ in Figure~\ref{fig:rationalmorita}.  Note that from the figure, $C, K, L$  have the following properties:   
\begin{enumerate}
\item There is a self-homeomorphism of $S$ fixing $C$ and sending $K$ to $L$ (namely, the Dehn twist on $M$).  
\item $K, L$ are separating curves on $S$ while $C$ is not.
\item $C$ is unknotted in $S^3$, $K = P(1,1,1) = T_{2,3}$ and $L = P(1,-1,1) = U$.  
\item Pushing $C$ off of $S$ has linking number 0 with each of $K$ and $L$.  
\item The surface framings of $K, L, C$ are all the Seifert framings, so that Torelli surgery for a $k$-fold Dehn twist corresponds to a $-1/k$-surgery.  
\item Pushing $C$ off $S$  and performing $-1/k$-surgery turns $K$ into $P(1,1,1 + 2k)$ and keeps $L$ unknotted.  
\end{enumerate}

Let $\phi, \psi, \eta$ denote negative Dehn twists along $C, K, L$ respectively.  Note that for all $k$
\begin{align*}
S^3_{\psi} &= S^3_{+1}(P(1,1,1)) = \Sigma(2,3,5), \\
S^3_{\eta} &= S^3_{\phi^k} = S^3, \\
S^3_{\psi\phi^k} &= S^3_{+1}(P(1,1,1+2k)).
\end{align*}
The Casson invariant of $S^3$ is 0.  The Casson invariant of $S^3_1(P(1,1,1+2k))$ is $k+1$ by \eqref{eq:casson-surgery}, since the Alexander polynomial of $P(1,1,1+2k)$ is $k t - (2k-1) + kt^{-1}$ for $k \geq 0$ and $|k|t - (2|k|+1) + |k|t^{-1}$ for $k < 0$.  Hence, we have the desired result.
\end{proof}

\bibliography{bib.bib}
\bibliographystyle{alpha}
\end{document}